%% file: main.tex
\documentclass[reqno]{amsart}
\usepackage[foot]{amsaddr}

\usepackage[hidelinks]{hyperref}

\usepackage{url}
\usepackage[nameinlink,capitalize]{cleveref}

\usepackage{amssymb}
\usepackage{amscd}
\usepackage{amsthm}
\usepackage{setspace}
\usepackage{enumerate}
\usepackage{graphicx}
\usepackage{mathtools}
\usepackage{tcolorbox} 
\usepackage{subcaption} 
\usepackage{siunitx}
\usepackage{tikz-cd}
\usepackage{bm,blkarray}
\numberwithin{equation}{section}
\usepackage{algorithm}
\usepackage{algpseudocode}
\usepackage{microtype}

\usepackage{mathtools}
\usepackage[tableposition=top]{caption}
\usepackage{booktabs,dcolumn}

\newcommand\R{\mathbb{R}}

\newcommand\eps{\varepsilon}

\usepackage{xcolor}      
\usepackage{listings}    

\definecolor{codegreen}{rgb}{0,0.6,0}
\definecolor{codegray}{rgb}{0.5,0.5,0.5}
\definecolor{codepurple}{rgb}{0.58,0,0.82}
\definecolor{white}{rgb}{1,1,1}

\lstdefinestyle{mystyle}{
    backgroundcolor=\color{white}, 
    commentstyle=\color{codegreen},
    keywordstyle=\color{magenta}, numberstyle=\tiny\color{codegray}, stringstyle=\color{codepurple}, basicstyle=\ttfamily\footnotesize, breakatwhitespace=false,  breaklines=true,                 
    captionpos=b,                    
    keepspaces=true,                 
    numbers=left,                    
    numbersep=5pt,                  
    showspaces=false,                
    showstringspaces=false,
    showtabs=false,                  
    tabsize=2,
    frame=single,
    rulecolor=\color{black},
    title=\lstname
}

\title[What is Jackson's constant?]{What is Jackson's constant?}
\author{Rikhav Shah}
\address{\scriptsize Department of Mathematics, Massachusetts Institute of Technology, Cambridge, MA, 02139 USA.}
\email{rdshah@mit.edu}
\author{John Urschel}
\email{urschel@mit.edu}
\author{Nicholas West}
\email{npwest00@mit.edu}

\newtheorem{theorem}{Theorem}[section]

\newtheorem{lemma}[theorem]{Lemma}

\input{preamble}

\renewcommand{\epsilon}{\eps}

\renewcommand{\paragraph}[1]{
\vspace{2.5mm}\noindent\textbf{#1. }
}

\newcommand{\jac}{J}
\newcommand{\riv}{R}

\begin{document}

\begin{abstract}
We prove a refinement of Jackson's theorem on the approximation of Lipschitz functions by trigonometric polynomials. Our result precisely characterizes the leading error term associated with Jackson's construction. We do the same for a related construction commonly used in the kernel polynomial method for spectral density estimation, which is slightly better than Jackson's construction in this respect.
\end{abstract}

\maketitle

\section{Overview and main results}

Theorems established by Dunham Jackson in the early 1900s regarding the approximation of Lipschitz functions by trigonometric and algebraic polynomials have undergone a resurgence in recent decades. This renewed interest is due in large part to their utility in large-scale eigenvalue computations for Hermitian matrices, most notably the kernel polynomial method for spectral density estimation and polynomial filtering in the Eigenvalues Slicing Library (EVSL) \cite{weisse2006kernel, lin2016approximating, li2019eigenvalues}. Indeed, Jackson's theorems and variants of his construction have been used repeatedly in recent research developing and analyzing algorithms for spectral density estimation \cite{braverman2022sublinear, musco2024sharper, chen2025randomized, musco2026spectraldensityestimationnormal}. In this note, we provide sharp error estimates for the constants associated with Jackson's original construction \cite{jackson1912approximation} and the more recent variant described by Rivlin \cite{rivlin1981introduction} and Wei{\ss}e et al.\ \cite{weisse2006kernel}. We find that the constant in Jackson's original construction is about \(9 \%\) larger than that of the more recent variants, suggesting that the latter possesses a slight edge for practical applications.

Throughout this note, we assume the given function \(f:\R\to\R\) is \(2\pi\)-periodic and \(L\)-Lipschitz continuous, i.e. there exists a constant \(L > 0\) such that
\[
    |f(x)-f(y)| \leq L|x - y|
    \quad \text{for all } x,y \in \mathbb{R}.
\]
A function is a \textit{trigonometric sum of degree} \(n\) if it is in the linear span of the functions
\[
\{1,\cos(x),\sin(x),\cdots,\cos(nx),\sin(nx)\}.
\]
Although we focus on the trigonometric problem here for simplicity, if \(f\) is defined on \([-1, 1]\) and is not periodic, one can convert all of the trigonometric polynomial approximation results into algebraic polynomial approximation results by approximating \(g(\theta) = f(\cos\theta)\) and mapping the result back via the substitution \(x = \cos\theta\), which sends \(\cos(k\theta)\) to the Chebyshev polynomial \(T_k(x)\). This reduction is standard, see for example \cite{rivlin1981introduction, devore1993constructive}. Consequently, the following results hold for algebraic polynomials approximating Lipschitz functions on $[-1, 1]$ as well as trigonometric polynomials approximating $2\pi$-periodic Lipschitz functions on $[-\pi, \pi]$.

Jackson's main theorems from \cite{jackson1912approximation} state the following.\footnote{Jackson also includes results analogous to \Cref{thm:jackson} on functions whose \((k-1)\)th derivative is \(L\)-Lipschitz, showing that the order of convergence is in general \(O(n^{-k})\), but this is beyond the scope of this note.}

\begin{theorem}[Jackson]\label{thm:jackson}
For each \(n\ge1\), there exists a trigonometric sum \(p\) of degree at most \(n\) satisfying
\begin{equation}\label{eq:jacksonorigbound}
    \max_{x \in [-\pi, \pi]} |p(x)-f(x)| \leq \dfrac{2.9L}{n}.
\end{equation}
\end{theorem}

Jackson's proof has several attractive properties for applications: it is \textit{constructive}, i.e. he shows explicitly how to convert a given function \(f\) into a trigonometric sum \(p\); it is \textit{positivity-preserving}, so that if \(f(x)\ge0\) everywhere then \(p(x)\ge0\) everywhere; and it is \textit{mass-preserving} in the sense that \(\int f(x)\wrt x = \int p(x) \wrt x\).
A related but distinct construction of Rivlin \cite{rivlin1981introduction}, derived independently by Wei{\ss}e et al.\ \cite{weisse2006kernel}, is also positivity-preserving, mass-preserving, and satisfies a guarantee similar to \cref{eq:jacksonorigbound} for some other constant in place of 2.9.\footnote{We note that Wei{\ss}e et al. refer to the kernel derived in their work as ``Jackson's kernel'' although technically the kernel is not identical to that used by Jackson or as presented in \cite{devore1993constructive}. It is, however, identical to the kernel derived in \cite{rivlin1981introduction} in proving a variant of Jackson's theorem.} We emphasize these latter two properties of positivity and mass preservation since they undergird the aforementioned application to the kernel polynomial method for approximating probability distributions. Positivity-preservation and mass-preservation ensure in particular that if \(f\) defines a probability density function on \([-\pi,\pi]\), then \(p\) will as well.

Jackson also showed that the constant 2.9 appearing in \cref{eq:jacksonorigbound} could not be pushed lower than \(\pi/2\) for any construction, irrespective of the positivity and normalization requirements, though he did not show that \(\pi/2\) was achievable. This was shown later for a different construction by Favard and independently by Akhiezer and Krein. More specifically, one can construct an optimal approximation satisfying the improved error bound
\[
    \max_{x \in [-\pi, \pi]} |p(x)-f(x)| \leq \dfrac{\pi L}{2(n+1)},
\]
which precisely matches Jackson's lower bound.
An exposition of this result is given, for example, in \cite{devore1993constructive}. Although Favard's construction is optimal, it is not positivity-preserving, which is significant, for example, in the approximation of probability densities.

The purpose of this note is to identify the exact error constants involved in the positivity-preserving constructions of Jackson and Rivlin. Before stating our theorems, we first state what these constructions are.
Define the following kernel functions:
\begin{alignat*}{2}
    &\jac_m(t) = \pare{\frac{\sin(mt/2)}{m\sin(t/2)}}^4\cdot\mathbf1_{[-\pi,\pi]}(t), \hspace{2em}
    && \riv_n(t) = \pare{\frac{1}{2} + \sum_{j=1}^n \rho_{j,n}\cos(jt)}\cdot\mathbf1_{[-\pi,\pi]}(t), \\
    & &&
    \rho_{j,n} =
    \frac{(n+2-j)\cos\pare{\frac{j\pi}{n+2}}
    + \sin\pare{\frac{j\pi}{n+2}}\cot\pare{\frac{\pi}{n+2}}}{n+2},
    \\
    &\hat\jac_m(t)=\frac1{\int_\R\jac_m(x)\wrt x}\jac_m(t),
    &&
    \hat\riv_n(t)=\frac1{\int_\R\riv_n(x)\wrt x}\riv_n(t).
\end{alignat*}
The kernel on the left, \(\jac_m\), we call Jackson's kernel and the kernel on the right, \(\riv_n\), we call Rivlin's kernel. The derivation of the coefficients \(\rho_{j,n}\) is provided in \cite{weisse2006kernel}. The convolution of two functions is defined as
\[
    (f*g)(x)=\int_\R f(x-t)g(t)\wrt t.
\]
Jackson's construction sets \(p=f*\hat\jac_m\) for \(m=\floor{n/2}+1\) and Rivlin's sets \(p=f*\hat\riv_n\). When restricted to $[-\pi, \pi]$, \(\riv_n\) is an even trigonometric sum of order \(n\), and \(\jac_m\) can be written as the square of an even trigonometric sum of order \((m-1)\), see \Cref{eq:jkernelexpansion}, so that \(\jac_m\) is an even trigonometric sum of order \(2(m-1) \leq n\). Since \(p\) is formed by convolution, in either case it will be a trigonometric sum of order \(n\).

Our main results provide precise control on leading terms of Jackson and Rivlin's approximations.

\begin{theorem}[Control of Jackson's kernel]
Let \(p=f*\hat\jac_m\) for \(n\geq 2\) even and \(m=(n/2)+1\). Then
\[
    \sup_{x\in\R} |f(x)-p(x)|
    \leq
    \dfrac{12\log 2}{\pi} \cdot \dfrac{L}{n+2}
    + O\pare{\frac{\log n}{n^3}}.
\]
Note \[\frac{12\log(2)}{\pi} = 2.6476\cdots.\] Moreover, the leading constant cannot be improved as it is attained when \(f(x) = L|x|\) periodically extended from \([- \pi, \pi]\) to the real line.
\end{theorem}

\begin{proof}
This is a direct consequence of \Cref{lem:momentbound} and \Cref{lem:jacksonbound}, proven in the second part of this note. Note that when the kernel \(\jac_m\) is used, it produces an approximation of degree \(n = 2(m-1)\). This causes the doubling from \(6 \log(2) / \pi\) in \Cref{lem:jacksonbound} to \(12 \log (2)/\pi\) in the present result, as well as the use of \(n+2\) in the denominator.  Explicitly, since \(1/m = 2/(n+2)\), \Cref{lem:jacksonbound} gives
\[
    \sup_{x\in\R} |f(x)-p(x)|
    \leq
    \dfrac{12\log 2}{\pi} \cdot \dfrac{L}{n+2}
    + \left(\dfrac{72\pi^3}{(n+2)^3} + \dfrac{288\log\pare{(n+2)/4}}{\pi(n+2)^3}\right)L.
\]
If \(n\) is odd, the result remains true with the substitution \(n+2 \to n+1\) in the denominator and the choice \(m = \floor{n/2}+1\).
\end{proof}

\begin{theorem}[Control of Rivlin's kernel]
Let \(p=f*\hat\riv_n\) and $n \geq 1$. Then
\[
    \sup_{x\in\R} |f(x)-p(x)|
    \leq
    \left(2 \int_0^\pi\dfrac{\sin u}{u} \wrt u - \dfrac{4}{\pi}\right)\dfrac{L}{n+2}
    + O\pare{\dfrac{1}{n^2}}.
\]
Note \[2 \int_0^\pi \frac{\sin u}{u} \wrt u - \frac{4}{\pi} = 2.4306 \cdots.\]
Moreover, the leading constant cannot be improved as it is attained when \(f(x) = L|x|\) periodically extended from \([- \pi, \pi]\) to the real line.
\end{theorem}

\begin{proof}
The result follows directly from \Cref{lem:momentbound} and \Cref{lem:weissebound}. Explicitly, \Cref{lem:weissebound} gives
\[
    \sup_{x\in\R} |f(x)-p(x)|
    \leq
    \left(2 \int_0^\pi\dfrac{\sin u}{u} \wrt u - \dfrac{4}{\pi}\right)\dfrac{L}{n+2}
    + \dfrac{50L}{\pi(n+2)^2}.
\]
\end{proof}

\section{Proofs}

Our first lemma establishes that the periodic extension of \(f(x)=L\abs x\) from \([-\pi,\pi]\) to all of \(\R\) is a worst-case function for an upper bound on \(\max_{x\in\R}\abs{f(x)-p(x)}\), and that the maximum is achieved for \(x=0\).

\begin{lemma}\label{lem:momentbound}
Let \(K \not\equiv 0\) be any non-negative kernel supported on \([-\pi,\pi]\) and set
\[
\hat K(t)=\frac1{\int K(t)\wrt t}K(t).
\]
Let \(f_0(x)\) be the periodic extension of \(L\abs x\) from \([-\pi,\pi]\) to all of \(\R\). Then
\begin{equation}
    \sup_{x\in\R}|(f*\hat K)(x)-f(x)|
    \le \abs{ (f_0*\hat K)(0)-f_0(0) }
    = L\dfrac{\int_{-\pi}^\pi \abs tK(t) \wrt t}{\int_{-\pi}^\pi K(t) \wrt t}.
\end{equation}
Equality is obtained at \(x=0\) for the triangle wave \(f(x) = L|x|\) periodically extended from \([-\pi, \pi]\) to the real line.
\end{lemma}

\begin{proof}
\begin{align*}
\abs{(f*\hat K)(x)-f(x)}
  &=\abs{\int_{-\pi}^\pi(f(x-t)-f(x))\hat K(t)\wrt t} \\
  &\le\int_{-\pi}^\pi\abs{f(x-t)-f(x)}\hat K(t)\wrt t \\
  &\le L\int_{-\pi}^\pi\abs t\hat K(t)\wrt t.
\end{align*}
Observe that equality is obtained at $x = 0$ for $f_0$ since $\abs{(f_0*\hat K)(0) - f_0(0)}=\abs{\int_{-\pi}^\pi f_0(-t)\hat K(t) \wrt t}=L\int_{-\pi}^\pi \abs t\hat K(t) \wrt t,$ 
which is the right-hand side above.
\end{proof}

Our goal is therefore to compute the ratio of \(\int_{-\pi}^\pi \abs t K \wrt t\) and \(\int_{-\pi}^\pi K \wrt t\) for \(K\in\set{\jac_m,\riv_n}\), which the next lemmas do. Our bounds establish the appropriate leading term of the error; however, in order to keep the presentation simple and brief, we have not attempted to optimize the final bounds.  

\begin{lemma}\label{lem:jacksonbound} Suppose $m \geq 2$. Then
\[
    \left|\dfrac{\int_{-\pi}^\pi\abs t\jac_m(t) \wrt t}{\int_{-\pi}^\pi\jac_m(t) \wrt t}
    -
    \dfrac{6\log (2)}{\pi}\cdot \dfrac{1}{m} \right|\leq 
    \frac{9 \pi^3}{m^3} + \frac{36 \log(m/2)}{ \pi m^3}.
\]
\end{lemma}

\begin{proof}
We will set
\[
A= m^4\int_0^\pi t\jac_m(t) \wrt t, \quad 
B= m^4\int_0^\pi\jac_m(t) \wrt t.
\]
Using the evenness of $\jac_m(t)$ and $|t|\jac_m(t)$ and canceling common factors in the numerator and denominator, the desired quantity is \(A/B\).

We first compute \(B\). We will use the identity
\begin{equation}\label{eq:jkernelexpansion}
    \dfrac{1}{2m}\left( \dfrac{\sin(mt/2)}{\sin(t/2)}\right)^{2}
    =
    \dfrac{1}{2}
    + \sum_{k=1}^{m-1}\left(1-\dfrac{k}{m}\right) \cos(kt)
\end{equation}
found in \cite[Page 203]{devore1993constructive}.\footnote{The formula in \cite[Page 203]{devore1993constructive} has a typo that includes \(k=0\) in the sum. We have corrected that here.}
Along with orthogonality of \(\cos(kx)\), we have
\begin{align*}
    B
    =
    \int_0^\pi \left(\dfrac{\sin(mx/2)}{\sin(x/2)} \right)^4 \wrt x
    &=
    4m^2 \left( \dfrac{\pi}{4}
    + \dfrac{\pi}{2}\sum_{k=1}^{m-1}\left(1-\dfrac{k}{m}\right)^2\right) \\
    &=
    4m^2 \left( \dfrac{\pi}{4}
    + \dfrac{\pi}{2}\dfrac{(m-1)(2m-1)}{6m}\right) \\
    &=
    \frac\pi3(2m^3+m).
\end{align*}

We now turn our attention to the numerator. First use the substitution \(u=mt\) and factor to obtain
\[
A=
m^2 \int_0^{m \pi}
u \pare{\dfrac{\sin(u/2)}{u/2}}^4
\pare{\dfrac{u/2m}{\sin(u/2m)}}^4
\wrt u.
\]
First using \(\abs{x/\sin(x)}\ge1\) and then \(\abs{\sin(x)/x}\le1/\abs x\),
\begin{align*}
A
&\ge
m^2 \int_0^{m \pi} u \pare{\dfrac{\sin(u/2)}{u/2}}^4 \wrt u \\
&=
m^2 \pare{4\log2-\int_{m \pi}^\infty u \pare{\dfrac{\sin(u/2)}{u/2}}^4 \wrt u} \\
&\ge
m^2 \pare{4\log2-\int_{m \pi}^\infty \frac{16}{u^3} \wrt u} \\
&=
m^2 \pare{4\log2-\frac8{\pi^2m^2}}.
\end{align*}
For an upper bound on \(A\), use \(\abs{{x}/{\sin (x)}}\le(1+6x^2)^{1/4}\) for \(\abs x\le \pi/2\):
\begin{equation}
    A
    \leq
    m^2\int_0^{m \pi}
    u \left( \dfrac{\sin(u/2)}{u/2}\right)^4
    \left(1+\dfrac{3u^2}{2m^2}\right) \wrt u
    \leq
    m^2\cdot4\log(2)
    + 24\int_0^{m\pi}\dfrac{\sin^4(u/2)}{u} \wrt u.
\end{equation}
The final integral may be bounded by
\begin{equation}
    \int_0^{m\pi}\dfrac{\sin^4(u/2)}{u} \wrt u
    \leq
    \int_0^{2\pi}\dfrac{(u/2)^4}{u} \wrt u
    + \int_{2\pi}^{m\pi}\dfrac{1}{u} \wrt u
    =
    \dfrac{\pi^4}{4}+\log(m/2).
\end{equation}
In particular, \(A \leq m^2\cdot4\log(2)+6\pi^4 + 24\log(m/2)\). Finally, we take the quotient:
\[
\frac AB
=
\dfrac{\int_0^\pi t\jac_m(t) \wrt t}{\int_0^\pi \jac_m(t) \wrt t}
\leq
\dfrac{m^2\cdot4\log(2)+6\pi^4 + 24\log(m/2)}{\dfrac{\pi}{3}(2m^3 + m)}
\leq
\dfrac{6 \log 2}{\pi} \cdot \dfrac{1}{m}
+
\frac{9\pi^3}{m^3} + \frac{36 \log(m/2)}{\pi m^3}.
\]
Similarly, using our lower bound on $A$, we have the stronger lower bound on the ratio
\[
    \frac AB \geq \frac{m^2 \cdot 4 \log 2 - 8/\pi^2}{\frac \pi 3 (2m^3 +m )} \geq  \frac{6 \log 2}{\pi} \cdot \frac{1}{m} - \frac{12}{\pi m^3}.
\]
\end{proof}

\begin{lemma}\label{lem:weissebound} Suppose $n \geq 1$. Then 
\[
    \left|\dfrac{\int_{-\pi}^\pi \abs t\riv_n(t) \wrt t}{\int_{-\pi}^\pi \riv_n(t) \wrt t}
    -
    \left(2\int_0^\pi \dfrac{\sin t}{t} \wrt t - \dfrac{4}{\pi}\right)\cdot \dfrac{1}{n+2}\right| \leq 
    \frac{50}{\pi(n+2)^2}.
\]
\end{lemma}

\begin{proof}
\(\riv_n\) is even, so it suffices to integrate over \(0\) to \(\pi\) instead.
The denominator can be computed exactly:
\[
\int_{-\pi}^\pi\riv_n(t)\wrt t=\pi
\]
since \(\int_{-\pi}^\pi1\wrt t=2\pi\) and \(\int_{-\pi}^\pi\cos(jt)\wrt t=0\) for \(j\ge 1\).
Integrating the numerator term by term gives
\begin{equation}\label{eq:rivnumerator}
\int_0^\pi t\riv_n(t) \wrt t
=
\dfrac{\pi^2}{4}
-
2\underset{j \; \textrm{odd}}{\sum_{1 \leq j \leq n}}
\dfrac{\rho_{j,n}}{j^2}
=
2\underset{j \; \textrm{odd}}{\sum_{1 \leq j \leq n}}
\dfrac{1 - \rho_{j,n}}{j^2}
+
2\underset{j \; \textrm{odd}}{\sum_{j > n}}
\dfrac{1}{j^2},
\end{equation}
where the first equality follows from direct integration of \(t\riv_n(t)\) according to its series definition and the second equality uses the fact that the sum of the odd reciprocal squares is equal to \(\pi^2/8\).

We claim that the first term may be rewritten
\begin{equation}
    2\underset{j \; \textrm{odd}}{\sum_{1 \leq j \leq n}}
    \dfrac{1 - \rho_{j,n}}{j^2}
    =
    \dfrac{1}{n+2} \int_{0}^{1} \psi(x) \wrt x
    + O\left(\dfrac{1}{n^2}\right), \quad \psi(x)
    \equiv
    \dfrac{1}{x^2}
    \left(1-(1-x)\cos(\pi x) - \dfrac{\sin(\pi x)}{\pi}\right).
\end{equation}
To see this, observe that
\[
    \frac{\pi}{3(n+2)^2} \leq \frac{1}{\pi} - \dfrac{\cot\left({\pi}/({n+2})\right)}{n+2}
    \leq 
    \frac{4}{\pi(n+2)^2},
\]
which follows from the observation $h^2/3 \leq 1-h\cot(h) \leq 4h^2/\pi^2$ for $h \in (0, \pi/2]$.
Setting
\[
\phi(x) \equiv  1-(1-x)\cos(\pi x) - \dfrac{\sin(\pi x)}{\pi},
\]
we then have by definition of the $\rho_{j, n}$
\[
    1 - \rho_{j,n}
    =
    \phi\left(\dfrac{j}{n+2}\right)
    + \delta_{j,n}, \quad
    \delta_{j,n} \equiv \sin\left(\frac{j\pi}{n+2}\right)\left(\frac{1}{\pi} - \frac {\cot\left({\pi}/({n+2})\right)}{{n+2}}\right).
\]
We now wish to estimate 
\[
     2\underset{j \; \textrm{odd}}{\sum_{1 \leq j \leq n}} \frac{\phi\left(\frac{j}{n+2}\right)+\delta_{j,n}}{j^2}.
\]
Observe that since $0 \leq \sin(j\pi/(n+2))\leq 1$ for $1 \leq j \leq n$; $\delta_{j,n} \leq 4\pi^{-1}(n+2)^{-2}$; and the sum over a finite number of odd reciprocal squares is less than or equal to $\pi^2/8$, 
\[
    0 \leq 2\underset{j \; \textrm{odd}}{\sum_{1 \leq j \leq n}}
    \dfrac{\delta_{j,n}}{j^2} \leq \frac{\pi}{(n+2)^2}.
\]
Moreover, we interpret the summation over $\phi(j/(n+2))/j^2 = (n+2)^{-2}\psi(j/(n+2))$ as a Riemann sum. We note $\psi(x)$ extends smoothly to \([0, 1]\) since \(\phi(0) = \phi'(0) = 0\), and elementary calculus yields $\max_{x \in [0, 1]}|\psi(x)|  \leq 5$ and $\max_{x \in [0, 1]}|\psi'(x)|  \leq 5$. We then write
\[
    2\underset{j \; \textrm{odd}}{\sum_{1 \leq j \leq n}}
    \dfrac{\phi\left(\dfrac{j}{n+2}\right)}{j^2}
    =
    \dfrac{2}{(n+2)^2}
    \underset{j \; \textrm{odd}}{\sum_{1 \leq j \leq n}}
    \psi\left(\dfrac{j}{n+2}\right).
\]
Observe that
\[
    \left|\int_0^1 \psi(x)dx - \dfrac{2}{n+2}
    \underset{j \; \textrm{odd}}{\sum_{1 \leq j \leq n}}
    \psi\left(\dfrac{j}{n+2}\right)\right| \leq \frac{2\max_{x \in [0, 1]}|\psi(x)|}{n+2} + \frac{ \max_{x \in [0, 1]}|\psi'(x)|}{n+2} \leq \frac{15}{n+2},
\]
where the bound is obtained by splitting the domain of integration into $[0, 1] = [0, a] \cup [a, b] \cup [b, 1]$ with $a = (n+2)^{-1}$ and $b = a+2 \lceil n/2 \rceil(n+2)^{-1}$. The first term bounds $\int_{[0, 1] \setminus [a,b]} \psi$, which is a set of measure at most $2(n+2)^{-1}$, and  the second bounds the error in the left-endpoint Riemann sum approximating $\int_a^b\psi$ with spacing $h = 2(n+2)^{-1}$, which has error at most $(h/2)(b-a) \max_{x \in [a, b]} |\psi'(x)|$ (see for example \cite{davis2007methods}). Moreover, integration by parts yields
\[
    \int_0^1 \psi(x) \wrt x
    =
    \pi \int_0^\pi\dfrac{\sin u}{u} \wrt u -3.
\]

The second term of \Cref{eq:rivnumerator} may be estimated for $n \geq 1$ by using 
\[
    \frac{1}{2}\int_j^{j+2} \frac{dx}{x^2} \leq \frac{1}{j^2} \leq \frac{1}{2} \int_{j-1}^{j+1}\frac{dx}{x^2}, \quad j \geq 3,
\]
so that 
\[
    \frac{1}{n+2} \leq 2\underset{j \; \textrm{odd}}{\sum_{j > n}}
    \dfrac{1}{j^2}
     \leq \frac{1}{n+2} + \frac{2}{n(n+2)} = \frac{1}{n}.
\]
Combining these bounds, we finally obtain
\[
    \left| \int_0^\pi t \riv_n(t) \wrt t
    -
    \dfrac{1}{n+2}
    \left(\pi \int_0^\pi \dfrac{\sin u}{u} \wrt u - 2\right)\right|
     \leq \frac{25}{(n+2)^2}.
\]
The final result follows from dividing by $ \int_0^\pi R_n(t) dt = \pi/2$.
\end{proof}

\section*{Acknowledgements}

This material is based upon work supported by the National Science Foundation under grant no. DMS-2513687. Nicholas was supported by the MathWorks fellowship.

\bibliographystyle{alpha}
\bibliography{main}

\end{document}

%% file: preamble.tex
\usepackage{multicol}
\usepackage{mleftright}
\usepackage{enumitem}
\usepackage{caption}
 \usepackage[nameinlink,capitalize]{cleveref}

\usepackage{refcount}

\newcommand{\wrt}{\,\textnormal d}

\newcommand{\abs}[1]{\mleft|#1\mright|}

\newcommand{\pare}[1]{\mleft(#1\mright)}

\newcommand{\set}[1]{{\left\{{#1}\right\}}}

\newcommand{\floor}[1]{\mleft\lfloor#1\mright\rfloor}

\makeatletter
\newtheorem*{rep@theorem}{\rep@title}
\newcommand{\newreptheorem}[2]{%
\newenvironment{rep#1}[1]{%
 \def\rep@title{#2 \ref{##1}}%
 \begin{rep@theorem}}%
 {\end{rep@theorem}}}
\makeatother

\newreptheorem{theorem}{Theorem}
\newreptheorem{coro}{Corollary}

\AddToHook{env/lemma/begin}{\crefalias{theorem}{lemma}}
\AddToHook{env/corollary/begin}{\crefalias{theorem}{corollary}}
\AddToHook{env/proposition/begin}{\crefalias{theorem}{proposition}}

\crefname{theorem}{Theorem}{Theorems}
\crefname{lemma}{Lemma}{Lemmas}
\crefname{corollary}{Corollary}{Corollaries}
\crefname{equation}{}{}
\crefname{subsection}{subsection}{subsections}
\crefname{proposition}{Proposition}{Propositions}